\documentclass[letterpaper, 10 pt, conference]{ieeeconf}

\IEEEoverridecommandlockouts
\usepackage[utf8]{inputenc}
\usepackage{amsmath,amssymb,amsfonts,mathtools}

\usepackage{amsthm}
\usepackage{xcolor}
\usepackage{graphicx}
\usepackage{float}
\usepackage{babel}
\usepackage{algorithm}
\usepackage{algorithmic}
\usepackage{bbm}
\usepackage{booktabs}
\usepackage{caption}
\usepackage{subcaption}
\usepackage{array}
\usepackage[hidelinks]{hyperref}

\newtheorem{theorem}{Theorem}
\newtheorem{lemma}{Lemma}
\newtheorem{proposition}{Proposition}
\newtheorem{corollary}{Corollary}
\newtheorem{assumption}{Assumption}
\newtheorem{remark}{Remark}
\newtheorem{definition}{Definition}

\newcommand{\E}{\mathbb{E}}
\newcommand{\R}{\mathbb{R}}
\newcommand{\norm}[1]{\left\|#1\right\|}

\newcommand{\tr}{\mathrm{tr}}
\newcommand{\Acl}{A_{\mathrm{cl}}}

\title{Finite-Sample Closed-Loop Stability of Model Predictive Path Integral Control for Linear Time-Invariant Systems}

\author{
Hyung-Jin Yoon$^{\dagger}$ and
Hunmin Kim$^{\ddagger}$
\thanks{$^{\dagger}$H.-J. Yoon is with the
Department of Mechanical and Nuclear Engineering, Tennessee Technological
University, Cookeville, TN, USA.}%
\thanks{$^{\ddagger}$H. Kim is with the School of Engineering, Department
of Electrical and Computer Engineering, Mercer University, Macon, GA, USA.}%
\thanks{This work was supported by internal funding at Tennessee
Technological University.}
}

\begin{document}
\maketitle

% -----------------------------------------------------------------------
\begin{abstract}
We establish finite-sample closed-loop stability guarantees for Model
Predictive Path Integral (MPPI) control applied to discrete-time Linear
Time-Invariant (LTI) systems with additive Gaussian process
disturbances.  The key observation is that, for unconstrained
LTI/quadratic systems with the DARE terminal cost, the exact
finite-horizon MPC law has the same first control action as the
infinite-horizon LQR law for every planning horizon.  Thus, finite-sample
MPPI can be analyzed as a stochastic perturbation of LQR.
First, we show that the MPPI control law approximates the LQR feedback
with high probability.  The approximation error decomposes into a Monte
Carlo term that decreases with the sample count and an infinite-sample
temperature bias that persists at finite temperature but vanishes as the
temperature is reduced.  The resulting constants are written in terms of
the horizon-dependent stacked cost matrices, making explicit that the
finite-sample certificate is parametrized by the selected planning
horizon.
Second, we use a Lyapunov perturbation argument to prove practical
exponential stability in expectation.  On sample paths that remain in a
compact Lyapunov sublevel set over a finite operating horizon, the
expected state norm decays exponentially up to three residual floors: a
process-noise floor, an MPPI approximation floor, and a confidence floor
from the per-step sampling failure probability.  The sufficient sample
threshold is explicit and computable from the DARE solution, LQR
stability margin, MPPI sampling parameters, temperature, and planning
horizon.  In the joint limit of infinite samples and vanishing
temperature bias, the result recovers the stochastic LQR stability
bound.
\end{abstract}

% -------------------------------------------------------
\section{Introduction}
\label{sec:intro}
% -------------------------------------------------------

Model Predictive Path Integral (MPPI) control~\cite{williams2017mppi,williams2018information}
is a sampling-based receding-horizon method that has achieved strong empirical performance
across robotics and autonomous systems, including off-road navigation~\cite{williams2016aggressive},
legged locomotion, and aerial vehicles.
Its central appeal is that it requires no gradient of the cost or dynamics: at each time step
it draws $M$ random control perturbations, rolls them out in parallel (GPU-accelerated), and
forms an importance-weighted average that approximates the information-theoretic optimal.
This gradient-free, massively parallel structure makes MPPI uniquely attractive for nonlinear
and non-smooth problems where classical gradient-based MPC solvers struggle.

Despite a growing body of work on MPPI theory, the question of \emph{closed-loop stability}
under receding-horizon execution remains largely unresolved.
In particular, it is unclear under what conditions the state $x_k$ remains bounded and converges
when MPPI is implemented with a finite sample count $M$ and subjected to persistent process
disturbances.
This is not merely an academic concern: without stability guarantees, practitioners cannot
reason systematically about how many samples are sufficient or how performance degrades as $M$
decreases.
The difficulty is that finite-sample approximation errors arise at every control update and
interact with stochastic disturbances over the horizon, making one-step optimization guarantees
insufficient for establishing long-term closed-loop behavior.

\subsection{Prior Work and the Remaining Gap}
\label{sec:related}

\textbf{MPPI foundations.}
The original MPPI derivation~\cite{williams2017mppi} frames control as minimization of a
KL-divergence between a controlled and an uncontrolled trajectory distribution, with the
importance-weighted update arising as the solution to this information-theoretic problem.
Williams et al.~\cite{williams2018information} extend this to a full information-theoretic MPC
framework.
Wagener et al.~\cite{wagener2019online} unify sampling-based MPC methods through online learning
with Bregman divergences.

\textbf{Approximation error and optimizer convergence.}
Yoon et al.~\cite{yoon2022sampling} established open-loop $O(M^{-1/2})$ sampling-complexity
bounds for Monte Carlo estimates of the path-integral optimal control, laying the groundwork for
quantitative convergence analysis.
Yi et al.~\cite{yi2024covo} (CoVO-MPC, L4DC~2024) provide the first convergence analysis of
MPPI as an optimizer, showing that the importance-weighted update contracts toward the optimal
control sequence at a linear rate for quadratic costs and characterizing the contraction rate as
a function of the sampling covariance $\Sigma_\epsilon$ and temperature $\lambda$.
Homburger et al.~\cite{homburger2025optimality} study optimality gaps in deterministic and
stochastic MPPI.
Fazlyab et al.~\cite{fazlyab2026model} interpret MPPI as preconditioned gradient descent on a
KL-regularized free-energy objective, connecting information-theoretic control to first-order
optimization.
Collectively, these works establish increasingly strong guarantees on the optimization step
performed by MPPI, but they do not address the stability of the resulting receding-horizon
closed-loop system.

\textbf{Robust MPPI and performance bounds.}
Gandhi et al.~\cite{gandhi2021robust} propose a Robust MPPI architecture with an augmented
nominal--actual state representation and derive a bound on free-energy growth as a function of
constraint violation level, tracking controller performance, and sampling error.
This gives a performance certificate for a specific architecture but does not constitute a
Lyapunov-based closed-loop stability proof.

\textbf{Contraction theory and CLF-based MPC.}
Lohmiller and Slotine~\cite{lohmiller1998contraction} establish contraction theory as a framework
for global convergence: any two trajectories contract at a uniform exponential rate, independently
of initial conditions.
Manchester and Slotine~\cite{manchester2017control} develop Control Contraction Metrics (CCMs)
that provide constructive synthesis conditions for contracting feedback policies.
Mayne et al.~\cite{mayne2000constrained} provide the classical MPC stability framework based on
a CLF terminal cost and a terminal invariant set; we replace the terminal set requirement with a
global CLF, removing the need for a terminal constraint set altogether.

\textbf{The open problem.}
A recent survey by Honda~\cite{honda2026model} identifies closed-loop
stability of path-integral MPC as a major open problem. Optimizer
convergence and closed-loop stability address fundamentally different
questions: the former asks whether a single MPPI planning step produces
a good control update for a fixed state, whereas the latter asks whether
the state trajectory generated by repeated receding-horizon execution
remains bounded over time. This paper addresses this gap for the
LTI/quadratic setting by combining finite-sample MPPI approximation
bounds with a Lyapunov perturbation argument.

\subsection{Contributions}

This paper provides, to the best of our knowledge, the first closed-loop
stability certificate for MPPI on LTI systems.  Specifically:
\begin{enumerate}
  \item \textbf{Finite-sample approximation bound (Lemma~\ref{lem:approx}):}
    We decompose the MPPI approximation error into two components: a
    finite-sample Monte Carlo error bounded by $\varepsilon_M(\eta)=
    O(M^{-1/2})$ with probability at least $1-\eta$, and an
    infinite-sample temperature bias $b_\infty(x_k,\bar{U})$ that
    persists as $M\to\infty$.  The bias is characterized in closed form
    via the bias gain $\kappa_\lambda$ (Proposition~\ref{prop:bias})
    and vanishes as $\lambda\to0$.
    This two-component decomposition is the finite-$M$ statement that
    CoVO-MPC leaves open.

  \item \textbf{Exponential stability in expectation
    (Theorem~\ref{thm:stability}):}
    The stopped process satisfies an unconditional Lyapunov bound for
    all $k\ge0$.
    On sample paths satisfying $\tau_R>T$, which
    occur with probability at least $1-\delta$, the unstopped process
    satisfies~\eqref{eq:state_practical_bound} for all $0\le k\le T$
    and all $M\ge M^*$,

  \item \textbf{Explicit sample threshold (Corollary~\ref{cor:mstar}):}
    \[
      M^* = O\!\left(
        \frac{\|P\|^2\|B\|^2\|\Acl\|^2}{\alpha_P^2}
        \log\frac{m}{\eta}
      \right)
    \]
    in terms of the LQR stability margin $\alpha_P$, system dimensions,
    and confidence level $1-\eta$, computable from the DARE solution.

  \item \textbf{ISS interpretation (Proposition~\ref{prop:iss}):}
    The main bound is recast as a practical Input-to-State Stability
    estimate with three explicit gains connecting the result to the ISS
    framework standard in robust MPC.
\end{enumerate}

The key technical ideas are threefold.  First, finite-sample MPPI is
interpreted as a perturbation of LQR with magnitude $O(M^{-1/2})$,
decomposed into a Monte Carlo component and a temperature-bias component.
Second, a Gaussian completing-the-square argument characterizes the
infinite-sample bias in closed form via $\kappa_\lambda\to0$ as
$\lambda\to0$.  Third, high-probability invariance of the Lyapunov
sublevel set $\Omega_R$ is established via a supermartingale argument
(Lemma~\ref{lem:invariance}), resolving the circularity between the
compact-set assumption and the concentration bounds.
Embedding these ingredients within the classical MPC stability framework
of Mayne et al.~\cite{mayne2000constrained} yields explicit stability and
sample-complexity guarantees.  In the limit $M\to\infty$, the result
recovers the standard stochastic LQR stability certificate.

% ============================
% === PROBLEM FORMULATION ====
% ============================
\section{Problem Formulation}\label{sec:problem}

Consider the discrete-time LTI system
\begin{equation}\label{eq:system}
x_{k+1}=Ax_k+Bu_k+w_k,
\end{equation}
where $x_k\in\R^n$, $u_k\in\R^m$, and the process disturbance
$w_k\sim\mathcal N(0,\Sigma_w)$ is i.i.d.\ with $\Sigma_w\succ0$.

At each time step, MPPI approximately solves the finite-horizon optimal
control problem
\begin{equation}\label{eq:cost}
\begin{aligned}
J(x_k,U)
&=
\sum_{i=0}^{N-1}
\left(
x_{k+i}^\top Qx_{k+i}
+
u_{k+i}^\top Ru_{k+i}
\right) \\
&+
x_{k+N}^\top Px_{k+N},
\end{aligned}
\end{equation}
where $U=\{u_k,\ldots,u_{k+N-1}\}$, $Q\succeq0$, $R\succ0$,
and $P\succ0$ is the stabilizing solution of the DARE.
For unconstrained linear-quadratic systems, this choice embeds the
infinite-horizon cost-to-go beyond the prediction horizon, causing the
finite-horizon optimum to coincide with the LQR feedback law and
providing the Lyapunov structure used in classical MPC stability
analysis.

The MPPI control update is
\begin{equation}\label{eq:mppi}
u_i^{\mathrm{MPPI}}
=
\bar u_i
+
\frac{\sum_{j=1}^{M} w^{(j)}\epsilon_i^{(j)}}
     {\sum_{j=1}^{M} w^{(j)}},
\end{equation}
where $\epsilon_i^{(j)}\sim\mathcal N(0,\Sigma_\epsilon)$ are i.i.d.\
sampling perturbations and
\begin{equation}\label{eq:weights}
w^{(j)}
=
\exp\!\left(
-\frac{1}{\lambda}
J(x_k,\bar U+\mathcal E^{(j)})
\right),
\qquad
\lambda>0.
\end{equation}
The applied control is the first element $u_k=u_0^{\mathrm{MPPI}}$,
and the nominal sequence $\bar U$ is updated according to the standard
receding-horizon shift.

The corresponding infinite-horizon LQR problem is
\[
\min_{\{u_k\}}
\sum_{k=0}^{\infty}
\left(
x_k^\top Qx_k
+
u_k^\top Ru_k
\right),
\]
whose optimal controller is
\begin{equation}\label{eq:lqrgain}
u_k^{\mathrm{LQR}}
=
-Kx_k,
\qquad
K=(R+B^\top PB)^{-1}B^\top PA,
\end{equation}
where $P\succ0$ satisfies the DARE
\begin{equation}\label{eq:dare}
P
=
Q+A^\top PA
-
A^\top PB(R+B^\top PB)^{-1}B^\top PA.
\end{equation}
Under standard stabilizability and detectability assumptions,
$A_{\mathrm{cl}}:=A-BK$ is Schur stable and $V(x)=x^\top Px$
satisfies
\begin{equation}\label{eq:lqrdec}
V(A_{\mathrm{cl}}x)-V(x)
=
-x^\top(Q+K^\top RK)x
\le
-\alpha_P\|x\|^2,
\end{equation}
where $\alpha_P:=\lambda_{\min}(Q+K^\top RK)>0$.

Since the finite-horizon optimum coincides with the LQR solution under
the above construction, the central question becomes how closely the
finite-sample MPPI update tracks this optimum.

\begin{assumption}[Stabilizability]\label{ass:stab}
  The pair $(A,B)$ is stabilizable.
\end{assumption}

\begin{assumption}[Detectability]\label{ass:detect}
  The pair $(A,Q^{1/2})$ is detectable.
\end{assumption}

\begin{remark}
  Assumptions~\ref{ass:stab}--\ref{ass:detect} are the minimal conditions
  under which the DARE~\eqref{eq:dare} has a unique positive-definite
  solution~\cite{anderson1990optimal}.  If $Q\succ 0$, detectability
  is automatic.
\end{remark}

\begin{assumption}[DARE Terminal Cost]\label{ass:dare}
  The terminal cost $P$ in~\eqref{eq:cost} is the unique
  positive-definite solution to the DARE~\eqref{eq:dare}.
\end{assumption}

This choice is canonical: $P$ is simultaneously the LQR optimal
cost-to-go and the unique matrix satisfying the CLF decrease
condition~\eqref{eq:lqrdec}, which is the terminal cost condition
required by the Mayne et al.\ framework~\cite{mayne2000constrained}.
Crucially, it also ensures that the finite-horizon optimum of
$J(x_k,U)$ coincides exactly with the infinite-horizon LQR solution
$u_k^{\mathrm{LQR}}=-Kx_k$ for any planning horizon $N\ge1$,
eliminating any horizon-truncation contribution to the approximation
error.

\begin{assumption}[MPPI Sampling]\label{ass:mppi}
  MPPI draws $M$ i.i.d.\ perturbation sequences
  $\epsilon_i^{(j)}\sim\mathcal{N}(0,\Sigma_\epsilon)$,
  $\Sigma_\epsilon\succ 0$, independent across time steps.
\end{assumption}

\begin{assumption}[Bounded Warm-Start Map]\label{ass:compact_mppi}
The MPPI nominal sequence is updated by a measurable warm-start map
\[
  \bar{U}_{k+1}
  =
  \Pi_{\mathcal{U}_N}\!\bigl(\mathcal{S}(\bar{U}_k)\bigr),
\]
where $\mathcal{S}:\mathbb{R}^{mN}\to\mathbb{R}^{mN}$ is the
receding-horizon shift operator, $\mathcal{U}_N\subset\mathbb{R}^{mN}$
is a compact set with diameter $D_{\mathcal{U}}<\infty$, and
$\Pi_{\mathcal{U}_N}$ denotes projection onto $\mathcal{U}_N$.
The projection may be implemented by clipping, saturation, or any
bounded warm-start rule, and guarantees $\bar{U}_k\in\mathcal{U}_N$
for all $k\ge0$.
\end{assumption}

\begin{remark}
Assumption~\ref{ass:compact_mppi} serves two purposes: it guarantees
$\|\bar{U}_k\|\le D_{\mathcal{U}}$ for all $k$, which bounds the
bias coefficients $\beta_0$ in Definition~\ref{def:bias} and the
bad-event constant $C_{\mathrm{bad}}$ in~\eqref{eq:Cbad_def}; and it
provides the compact domain $\mathcal{U}_N$ over which the uniform
lower bound $\underline{Z}$ in Lemma~\ref{lem:weight_moments} is
attained.
\end{remark}

\begin{assumption}[Bounded MPPI Update]\label{ass:bounded_mppi}
The implemented MPPI control satisfies
\[
  \|u_k^{\mathrm{MPPI}}\| \le \bar{u} < \infty
\]
almost surely for some constant $\bar{u}$.
This is enforced in practice by actuator saturation, truncated Gaussian
sampling of $\epsilon_i^{(j)}$, or the projection
$\Pi_{\mathcal{U}_N}$ of Assumption~\ref{ass:compact_mppi}, all of
which bound the applied control almost surely.
\end{assumption}

\begin{remark}
Assumption~\ref{ass:bounded_mppi} is required to make the bad-event
Lyapunov bound mathematically clean.  Without it, Gaussian perturbations
are unbounded and the step
$\mathbb{E}[\tilde{V}_{k+1}\mathbf{1}_{\mathcal{G}_k^c}\mid\mathcal{F}_k]
\le C_{\mathrm{bad}}\eta$
would require a Cauchy--Schwarz argument yielding a $\sqrt{\eta}$ rather
than $\eta$ residual at the Lyapunov level.
Assumption~\ref{ass:bounded_mppi} is satisfied by every practical MPPI
implementation that clips perturbations or saturates the applied control.
\end{remark}

% =====================
% === MAIN RESULTS ====
% =====================
\section{Main Results}\label{sec:main}

This section establishes that finite-sample MPPI is a stochastic
approximation of the corresponding infinite-sample MPPI update.  The
difference between infinite-sample MPPI and LQR appears explicitly as a
bias term.  Stability follows whenever this bias and the finite-sample
error are small enough relative to the LQR Lyapunov decrease margin.

\subsection{Finite-Sample Approximation of the Infinite-Sample MPPI Update}

For a state $x_k$ and nominal control sequence
$\bar U\in\mathbb R^{mN}$, let
$\mathcal E \sim \mathcal N(0,I_N\otimes\Sigma_\epsilon)$
denote the stacked MPPI perturbation sequence and let
$\epsilon_0\in\mathbb R^m$ denote its first control block.  Define the
unnormalized MPPI weight
\[
  w(x_k,\bar U,\mathcal E)
  :=
  \exp\!\left(
    -\frac{1}{\lambda}
    J(x_k,\bar U+\mathcal E)
  \right).
\]
The infinite-sample MPPI perturbation mean is
\[
  \mu_\infty(x_k,\bar U)
  :=
  \frac{\mathbb E[w\epsilon_0]}{\mathbb E[w]},
\]
and the corresponding infinite-sample MPPI control is
$u_k^\infty := \bar u_0+\mu_\infty(x_k,\bar U)$.
The infinite-sample MPPI bias relative to LQR is
$b_\infty(x_k,\bar U) := \|u_k^\infty-u_k^{\mathrm{LQR}}\|$.

\begin{lemma}[Finite Weighted Moments in the LTI/Quadratic Case]
\label{lem:weight_moments}
Suppose Assumptions~\ref{ass:stab}--\ref{ass:dare} hold and
$\Sigma_\epsilon\succ0$.  For every compact set
$\mathcal X\subset\mathbb R^n$ and every bounded set of nominal control
sequences $\mathcal U_N\subset\mathbb R^{mN}$, there exist constants
$\underline Z>0$ and $C_\epsilon<\infty$ such that, for all
$x\in\mathcal X$ and all $\bar U\in\mathcal U_N$,
\begin{align}
  \mathbb E[w] &\ge \underline Z, \label{eq:weight_lower}\\
  \mathbb E[w^2] &\le 1, \label{eq:weight_second}\\
  \mathbb E[\|w\epsilon_0\|^2] &\le C_\epsilon. \label{eq:weighted_moment}
\end{align}
\end{lemma}

\begin{proof}
In the LTI/quadratic setting, the finite-horizon cost has the quadratic
form
\[
  J(x,U)
  =
  U^\top H U
  +
  2x^\top F U
  +
  x^\top Gx,
\]
where $H\succ0$, $F$ and $G$ are constant matrices determined by
$(A,B,Q,R,P,N)$, and $G\succeq0$.  Substituting $U=\bar U+\mathcal E$
gives
\[
  J(x,\bar U+\mathcal E)
  =
  \mathcal E^\top H\mathcal E
  +
  2(H\bar U+F^\top x)^\top\mathcal E
  +
  J(x,\bar U).
\]
Since stage and terminal costs are nonneg\-ative, $J(x,U)\ge0$, hence
$0<w\le1$.
Therefore $\mathbb E[w^2]\le1$ and
$\mathbb E[\|w\epsilon_0\|^2]\le\mathbb E[\|\epsilon_0\|^2]
=\tr(\Sigma_\epsilon)<\infty$,
giving \eqref{eq:weight_second}--\eqref{eq:weighted_moment} with
$C_\epsilon=\tr(\Sigma_\epsilon)$.

For \eqref{eq:weight_lower}, the map
$(x,\bar U)\mapsto\mathbb E[\exp(-J(x,\bar U+\mathcal E)/\lambda)]$
is continuous (since $J$ is polynomial and $w$ is bounded by $1$) and
strictly positive.  It therefore attains a positive minimum
$\underline Z>0$ on the compact set $\mathcal X\times\overline{\mathcal U_N}$,
where $\overline{\mathcal U_N}$ denotes the closure of $\mathcal U_N$.
\end{proof}

\begin{lemma}[Finite-Sample MPPI Concentration Around
the Infinite-Sample Update]
\label{lem:approx_infinite}
Suppose Assumptions~\ref{ass:stab}--\ref{ass:mppi} hold and
$\Sigma_\epsilon\succ0$.  Fix a compact set
$\mathcal X\subset\mathbb R^n$ and a bounded set
$\mathcal U_N\subset\mathbb R^{mN}$.
Then, for any $\eta\in(0,1)$, there exist constants
$C_{\mathcal X,\mathcal U}>0$ and $M_0(\eta)$ such that, for all
$M\ge M_0(\eta)$, all $x_k\in\mathcal X$, and all
$\bar U\in\mathcal U_N$,
\begin{equation}\label{eq:finite_mppi_infinite_error}
  \|u_k^{\mathrm{MPPI}}-u_k^\infty\|
  \le
  \varepsilon_M(\eta)
  :=
  C_{\mathcal X,\mathcal U}
  \sqrt{\frac{\log(4m/\eta)}{M}}
\end{equation}
with probability at least $1-\eta$.
\end{lemma}

\begin{proof}
Let $\hat Z_M:=\frac{1}{M}\sum_{j=1}^M w^{(j)}$,
$Z:=\mathbb E[w]$,
$\hat Y_M:=\frac{1}{M}\sum_{j=1}^M w^{(j)}\epsilon_0^{(j)}$,
and $Y:=\mathbb E[w\epsilon_0]$.
Then $\hat\mu_M=\hat Y_M/\hat Z_M$,
$\mu_\infty=Y/Z$, and
$u_k^{\mathrm{MPPI}}-u_k^\infty=\hat\mu_M-\mu_\infty$.

By Lemma~\ref{lem:weight_moments}, $Z\ge\underline Z>0$ uniformly.
Since $0<w\le1$, Hoeffding's inequality gives
$\mathbb P(\hat Z_M\le\underline Z/2)
\le\exp(-M\underline Z^2/2)$,
so the event $\mathcal D_M:=\{\hat Z_M\ge\underline Z/2\}$ holds with
probability at least $1-\eta/2$ for
$M\ge(2/\underline Z^2)\log(4/\eta)$.

Because $0<w\le1$ and $\epsilon_0$ is Gaussian, each coordinate of
$w\epsilon_0$ is sub-Gaussian.  A union bound over the $m$ coordinates
gives
\[
  \|\hat Y_M-Y\|
  \le
  C_Y\sqrt{\frac{\log(4m/\eta)}{M}}
\]
with probability at least $1-\eta/2$.

On the intersection of these two events,
\[
  \hat\mu_M-\mu_\infty
  =
  \frac{\hat Y_M-Y}{\hat Z_M}
  +
  Y\!\left(
    \frac{1}{\hat Z_M}-\frac{1}{Z}
  \right).
\]
Since $\hat Z_M\ge\underline Z/2$, $Z\ge\underline Z$, and
$\|Y\|\le\mathbb E[\|\epsilon_0\|]<\infty$, the union bound gives
$\|\hat\mu_M-\mu_\infty\|\le C_{\mathcal X,\mathcal U}
\sqrt{\log(4m/\eta)/M}$
with probability at least $1-\eta$.
The constants $\underline{Z}$, $C_Y$, and $C_{\mathcal{X},\mathcal{U}}$
are uniform over $\mathcal{X}\times\mathcal{U}_N$ by
Lemma~\ref{lem:weight_moments} and the compactness argument therein;
for each fixed $(x_k,\bar{U})$ the concentration event holds with
probability at least $1-\eta$.
\end{proof}

\begin{lemma}[Finite-Sample MPPI Approximation of LQR]
\label{lem:approx}
Under the same conditions as Lemma~\ref{lem:approx_infinite}, for any
$\eta\in(0,1)$ and all $M\ge M_0(\eta)$,
\begin{equation}\label{eq:mppi_lqr_with_bias}
  \|u_k^{\mathrm{MPPI}}-u_k^{\mathrm{LQR}}\|
  \le
  b_\infty(x_k,\bar U)
  +
  \varepsilon_M(\eta)
\end{equation}
with probability at least $1-\eta$, uniformly on
$\mathcal X\times\mathcal U_N$.
\end{lemma}

\begin{proof}
By the triangle inequality,
$\|u_k^{\mathrm{MPPI}}-u_k^{\mathrm{LQR}}\|
\le\|u_k^{\mathrm{MPPI}}-u_k^\infty\|
+\|u_k^\infty-u_k^{\mathrm{LQR}}\|$.
The first term is bounded by Lemma~\ref{lem:approx_infinite} and the
second is $b_\infty(x_k,\bar U)$.
\end{proof}

% -----------------------------------------------------------------------
\begin{proposition}[Infinite-Sample MPPI Bias in the LTI/Quadratic Case]
\label{prop:bias}
Suppose Assumptions~\ref{ass:stab}--\ref{ass:mppi} hold.
For the LTI system~\eqref{eq:system} with quadratic
cost~\eqref{eq:cost}, the finite-horizon cost is quadratic in the
control perturbation:
\begin{equation}\label{eq:cost_quadratic}
  J(x_k,\bar{U}+\mathcal{E})
  =
  \mathcal{E}^\top H\mathcal{E}
  +2v(x_k,\bar{U})^\top\mathcal{E}
  +J(x_k,\bar{U}),
\end{equation}
where $H\succ0$ is the cost Hessian determined by $(A,B,Q,R,P,N)$ and
$v(x_k,\bar{U}):=H\bar{U}+F^\top x_k$.
Under Assumption~\ref{ass:dare}, the finite-horizon optimizer is
\begin{equation}\label{eq:Ustar}
  U^*(x_k)
  :=
  \arg\min_U J(x_k,U)
  =
  -H^{-1}F^\top x_k,
\end{equation}
so $v(x_k,\bar{U})=H(\bar{U}-U^*(x_k))$ and $v=0$ if and only if
$\bar{U}=U^*(x_k)$.
The tilted distribution $w(\mathcal{E})\cdot p(\mathcal{E})$ is Gaussian
with mean
\begin{equation}\label{eq:mustar}
  \mu^*
  =
  -\frac{2}{\lambda}\Sigma_* v(x_k,\bar{U}),
  \qquad
  \Sigma_*
  :=
  \left(
    \Sigma_\epsilon^{-1}+\frac{2H}{\lambda}
  \right)^{-1},
\end{equation}
so $u_k^\infty=\bar{u}_0+\mu^*_0$, where $\mu^*_0\in\mathbb{R}^m$
is the first control block of $\mu^*$.
Since $u_k^{\mathrm{LQR}}=U^*_0(x_k)$ under Assumption~\ref{ass:dare},
the infinite-sample bias satisfies
\begin{equation}\label{eq:bias_decomp}
  u_k^\infty - u_k^{\mathrm{LQR}}
  =
  \left(
    S_0 - \frac{2}{\lambda}[\Sigma_* H]_0
  \right)
  (\bar{U}-U^*(x_k)),
\end{equation}
where $S_0\in\mathbb{R}^{m\times mN}$ is the selection matrix
extracting the first control block and
$[\Sigma_* H]_0\in\mathbb{R}^{m\times mN}$ is the first block row
of $\Sigma_* H$.
Define the bias gain
\begin{equation}\label{eq:kappa}
  \kappa_\lambda
  :=
  \left\|
    S_0 - \frac{2}{\lambda}[\Sigma_* H]_0
  \right\|.
\end{equation}
The infinite-sample bias satisfies:
\begin{enumerate}
  \item \textbf{(Zero bias at LQR nominal.)}
    If $\bar{U}=U^*(x_k)$, then $v=0$, $\mu^*=0$, and
    $b_\infty(x_k,\bar{U})=0$.

  \item \textbf{(Explicit bias bound.)}
    For general $\bar{U}\in\mathcal{U}_N$ and $x_k\in\Omega_R$,
    \begin{equation}\label{eq:bias_explicit}
      b_\infty(x_k,\bar{U})
      \le
      \beta_\infty\|x_k\|+\beta_0,
    \end{equation}
    where
    \begin{equation}\label{eq:beta_coeffs}
      \beta_\infty
      :=
      \kappa_\lambda\|H^{-1}F^\top\|,
      \qquad
      \beta_0
      :=
      \kappa_\lambda D_{\mathcal{U}},
    \end{equation}
    and $D_{\mathcal{U}}$ is the diameter of $\mathcal{U}_N$.

  \item \textbf{(Small-gain condition is satisfiable via $\lambda\to0$.)}
    Define
    $\Phi(\beta):=2\|P\|\|B\|\|A_{\mathrm{cl}}\|\beta
    +\|P\|\|B\|^2\beta^2$.
    As $\lambda\to0$,
    $\Sigma_*\sim(\lambda/2)H^{-1}$,
    so $\frac{2}{\lambda}\Sigma_* H\to I$ and
    $\kappa_\lambda\to0$.
    Hence $\beta_\infty\to0$ and the small-gain
    condition~\eqref{eq:bias_small_gain} is satisfied for all
    $\lambda\le\lambda^*$, where
    \begin{equation}\label{eq:lambdastar}
      \lambda^*
      :=
      \sup\bigl\{
        \lambda>0
        :
        \Phi(\beta_\infty(\lambda))
        \le\tfrac{\alpha_P}{2}
      \bigr\},
    \end{equation}
    which is strictly positive and computable from
    $(A,B,Q,R,P,N,\Sigma_\epsilon,\alpha_P)$.
\end{enumerate}
\end{proposition}

\begin{proof}
The quadratic form~\eqref{eq:cost_quadratic} follows by substituting
$U=\bar{U}+\mathcal{E}$ into~\eqref{eq:cost}.
Under Assumption~\ref{ass:dare}, the first-order optimality condition
$HU^*+F^\top x_k=0$ gives~\eqref{eq:Ustar}, so
$v=H(\bar{U}-U^*)$.
The product $w(\mathcal{E})\cdot p(\mathcal{E})$ has log-density
\[
\begin{aligned}
  &-\frac{1}{2}\mathcal{E}^\top\Sigma_\epsilon^{-1}\mathcal{E}
  -\frac{1}{\lambda}
  \bigl(
    \mathcal{E}^\top H\mathcal{E}
    +2v^\top\mathcal{E}
  \bigr)
  +\text{const} \\
  &\quad=
  -\frac{1}{2}(\mathcal{E}-\mu^*)^\top\Sigma_*^{-1}(\mathcal{E}-\mu^*)
  +\text{const},
\end{aligned}
\]
obtained by completing the square, identifying~\eqref{eq:mustar}.

\textit{Part~1.}
If $\bar{U}=U^*$, then $v=0$, $\mu^*=0$, and
$u_k^\infty=\bar{u}_0=U^*_0=u_k^{\mathrm{LQR}}$.

\textit{Part~2.}
Substituting $v=H(\bar{U}-U^*)$ into $\mu^*=-\frac{2}{\lambda}\Sigma_* v$
and taking the first block,
$\mu^*_0=-\frac{2}{\lambda}[\Sigma_* H]_0(\bar{U}-U^*)$.
Since $u_k^\infty-u_k^{\mathrm{LQR}}=S_0(\bar{U}-U^*)+\mu^*_0$,
substituting gives~\eqref{eq:bias_decomp}.  Therefore
$b_\infty\le\kappa_\lambda\|\bar{U}-U^*\|
\le\kappa_\lambda(\|H^{-1}F^\top\|\|x_k\|+D_{\mathcal{U}})$,
giving~\eqref{eq:bias_explicit}--\eqref{eq:beta_coeffs}.

\textit{Part~3.}
By the Sherman--Morrison--Woodbury identity,
$\Sigma_*=\Sigma_\epsilon
-\Sigma_\epsilon(\frac{\lambda}{2}H^{-1}+\Sigma_\epsilon)^{-1}\Sigma_\epsilon$.
As $\lambda\to0$, $\Sigma_*=(\lambda/2)H^{-1}+O(\lambda^2)$,
so $\frac{2}{\lambda}\Sigma_* H=I+O(\lambda)\to I$ and
$\kappa_\lambda\to0$.
Since $\Phi(\beta_\infty(\lambda))$ is continuous and vanishes as
$\lambda\to0$, $\lambda^*$ is strictly positive.
\end{proof}

\begin{remark}
The $\|\Sigma_\epsilon\|\to0$ route does \emph{not} make $b_\infty\to0$.
As $\|\Sigma_\epsilon\|\to0$, $\Sigma_*\to0$ so
$\frac{2}{\lambda}\Sigma_* H\to0$ and $\kappa_\lambda\to\|S_0\|=1$.
Small sampling covariance prevents exploration away from the nominal
sequence but does not drive the infinite-sample update toward
$u_k^{\mathrm{LQR}}$.  Only $\lambda\to0$ achieves $b_\infty\to0$.
\end{remark}

\begin{remark}[Role of the Planning Horizon]
\label{rem:horizon}
Under Assumption~\ref{ass:dare}, the deterministic finite-horizon
optimizer has the same first control component as the infinite-horizon
LQR law for every horizon $N\ge1$.  In particular, if
$U_N^*(x)$ denotes the $N$-step optimal control sequence for the
cost~\eqref{eq:cost}, then its first block satisfies
$[U_N^*(x)]_0=-Kx=u_k^{\mathrm{LQR}}$ independently of $N$.
Thus the DARE terminal cost removes any horizon-truncation error in the
nominal optimal feedback law.

This does not imply that the finite-sample MPPI certificate is
independent of $N$.  The stacked Hessian $H_N$, the linear term $F_N$,
the sampling dimension $mN$, the compact warm-start set
$\mathcal U_N$, and the concentration constant
$C_{\mathcal X,\mathcal U}$ generally depend on the planning horizon.
Consequently, the bias coefficients
$\beta_{\infty}$, $\beta_0$, the finite-sample error
$\varepsilon_M(\eta)$, and the sufficient sample threshold $M^*$ may
depend on $N$.  The horizon-independent part of the result is the
nominal LQR feedback recovered by the exact finite-horizon optimizer,
not the MPPI sample complexity.
\end{remark}

% -----------------------------------------------------------------------
\subsection{Telescoping Value Function Decrease}
% -----------------------------------------------------------------------

\begin{lemma}[Telescoping Decrease]\label{lem:telescope}
Under Assumptions~\ref{ass:stab}--\ref{ass:dare}, with $V(x)=x^\top Px$,
\begin{equation}\label{eq:telescope}
  V(A_{\mathrm{cl}}x)-V(x)
  =
  -x^\top(Q+K^\top RK)x
  \le
  -\alpha_P\|x\|^2,
\end{equation}
where $A_{\mathrm{cl}}:=A-BK$, $K=(R+B^\top PB)^{-1}B^\top PA$,
and $\alpha_P:=\lambda_{\min}(Q+K^\top RK)>0$.
\end{lemma}

\begin{proof}
The DARE identity gives
$A_{\mathrm{cl}}^\top P A_{\mathrm{cl}}=P-Q-K^\top RK$, so
$V(A_{\mathrm{cl}}x)-V(x)=-x^\top(Q+K^\top RK)x\le-\alpha_P\|x\|^2$.
\end{proof}

% -----------------------------------------------------------------------
\subsection{Closed-Loop Practical Stability}
% -----------------------------------------------------------------------

Write $u_k^{\mathrm{MPPI}}=u_k^{\mathrm{LQR}}+d_k$ where
$d_k:=u_k^{\mathrm{MPPI}}-u_k^{\mathrm{LQR}}$.
By Lemma~\ref{lem:approx}, with probability at least $1-\eta$,
$\|d_k\|\le b_\infty(x_k,\bar U_k)+\varepsilon_M(\eta)$.

Fix a Lyapunov sublevel set
\begin{equation}\label{eq:OmegaR}
  \Omega_R
  :=
  \{x\in\mathbb{R}^n : V(x)\le R\},
  \qquad
  R>0,
\end{equation}
where $R$ will be chosen in Lemma~\ref{lem:invariance} to guarantee
trajectories remain in $\Omega_R$ with probability at least $1-\delta$.

\begin{definition}[Bias Coefficients and Error Term]\label{def:bias}
For the LTI system~\eqref{eq:system} with DARE terminal
cost~\eqref{eq:dare}, define the bias gain $\kappa_\lambda$
as in~\eqref{eq:kappa} and the bias coefficients
\begin{equation}\label{eq:beta_def}
  \beta_\infty
  :=
  \kappa_\lambda\|H^{-1}F^\top\|,
  \qquad
  \beta_0
  :=
  \kappa_\lambda D_{\mathcal{U}},
\end{equation}
where $H$, $F$, and $D_{\mathcal{U}}$ are as in
Proposition~\ref{prop:bias}, and the composite error term
\begin{equation}\label{eq:eM}
  e_M(\eta) := \beta_0 + \varepsilon_M(\eta).
\end{equation}
By Proposition~\ref{prop:bias} Part~2,
$b_\infty(x,\bar{U})\le\beta_\infty\|x\|+\beta_0$
for all $x\in\Omega_R$ and $\bar{U}\in\mathcal{U}_N$, so on the good
event $\mathcal{G}_k:=\{\|d_k\|\le\beta_\infty\|x_k\|+e_M(\eta)\}$
of Lemma~\ref{lem:approx},
\begin{equation}\label{eq:dk_bound}
  \|d_k\|
  \le
  \beta_\infty\|x_k\|+e_M(\eta).
\end{equation}
\end{definition}

For the statements below, let $\alpha:=\alpha_P/\lambda_{\max}(P)$
and define
\begin{equation}\label{eq:L_def}
  L
  :=
  2\|P\|\|B\|\|A_{\mathrm{cl}}\|
  +2\|P\|\|B\|^2\beta_\infty,
\end{equation}
\begin{equation}\label{eq:Ce_def}
  C_e
  :=
  \frac{L^2}{\alpha_P}+\|P\|\|B\|^2,
  \qquad
  C_w^{(0)}
  :=
  \tr(P\Sigma_w),
\end{equation}
\begin{equation}\label{eq:Cbad_def}
  C_{\mathrm{bad}}
  :=
  \lambda_{\max}(P)
  \Bigl(
    \|A\|^2\sup_{x\in\Omega_R}\|x\|^2
    +\|B\|^2\bar{u}^2
  \Bigr)
  +\tr(P\Sigma_w),
\end{equation}
and the residual
\begin{equation}\label{eq:Delta_def}
  \Delta
  :=
  C_w^{(0)}+C_e e_M(\eta)^2+C_{\mathrm{bad}}\eta.
\end{equation}

\begin{remark}
The definition of $C_{\mathrm{bad}}$ uses Assumption~\ref{ass:bounded_mppi}
to bound $\|u_k^{\mathrm{MPPI}}\|\le\bar{u}$ almost surely.
On the bad event $\mathcal{G}_k^c$, this gives
$\|Ax_k+Bu_k^{\mathrm{MPPI}}+w_k\|^2\le(\|A\|\sqrt{R/\lambda_{\min}(P)}
+\|B\|\bar{u})^2+\tr(\Sigma_w)$ almost surely, so
$\tilde{V}_{k+1}\le C_{\mathrm{bad}}$ almost surely and the step
$\mathbb{E}[\tilde{V}_{k+1}\mathbf{1}_{\mathcal{G}_k^c}\mid\mathcal{F}_k]
\le C_{\mathrm{bad}}\eta$ is valid.
\end{remark}

% -----------------------------------------------------------------------
\begin{lemma}[High-Probability Finite-Horizon Invariance of $\Omega_R$]
\label{lem:invariance}
Suppose Assumptions~\ref{ass:stab}--\ref{ass:bounded_mppi} hold, the
bias coefficients $\beta_\infty$ and $\beta_0$ are as in
Definition~\ref{def:bias}, and the small-gain
condition~\eqref{eq:bias_small_gain} is satisfied.
For any $\delta\in(0,1)$ and any finite horizon $T\ge0$, choose
\begin{equation}\label{eq:R_choice}
  R
  \ge
  \frac{V(x_0)}{\delta}+\frac{2\Delta}{\alpha},
\end{equation}
where $\Delta$ and $\alpha$ are as in~\eqref{eq:Delta_def}
and above~\eqref{eq:L_def}.
Then the exit time
$\tau_R:=\inf\{k\ge0:x_k\notin\Omega_R\}$
satisfies
\begin{equation}\label{eq:exit_prob}
  \mathbb{P}(\tau_R\le T)
  \le
  \delta,
\end{equation}
i.e., $x_k\in\Omega_R$ for all $0\le k\le T$ with probability at least
$1-\delta$.
\end{lemma}

\begin{proof}
Define the stopped process $\tilde x_k:=x_{k\wedge\tau_R}$ and
$\tilde V_k:=V(\tilde x_k)$.
On the event $\{k<\tau_R\}$, $x_k\in\Omega_R$ and applying
Lemma~\ref{lem:telescope} and bound~\eqref{eq:dk_bound} from
Definition~\ref{def:bias} gives the one-step drift condition
\[
  \mathbb{E}[\tilde V_{k+1}\mid\mathcal{F}_k]
  \le
  \left(1-\frac{\alpha}{2}\right)\tilde V_k+\Delta.
\]
On $\{k\ge\tau_R\}$, $\tilde V_{k+1}=\tilde V_k$ and the drift is zero,
so the inequality holds trivially.
Unrolling the recursion for any fixed $T\ge0$,
\begin{equation}\label{eq:VT_bound}
  \mathbb{E}[\tilde V_T]
  \le
  \left(1-\frac{\alpha}{2}\right)^T V(x_0)
  +\frac{2\Delta}{\alpha}
  \le
  V(x_0)+\frac{2\Delta}{\alpha}.
\end{equation}
Since $\{\tau_R\le T\}\subseteq\{\tilde V_T\ge R\}$
(if the trajectory has already exited $\Omega_R$ by time $T$, then
$\tilde V_T=V(x_{\tau_R})\ge R$ by definition of $\tau_R$),
Markov's inequality applied to the fixed-time bound~\eqref{eq:VT_bound}
gives
\[
  \mathbb{P}(\tau_R\le T)
  \le
  \mathbb{P}(\tilde V_T\ge R)
  \le
  \frac{\mathbb{E}[\tilde V_T]}{R}
  \le
  \frac{V(x_0)+2\Delta/\alpha}{R}.
\]
Under~\eqref{eq:R_choice},
$R\ge V(x_0)/\delta+2\Delta/\alpha$, so
\[
  \frac{V(x_0)+2\Delta/\alpha}{R}
  \le
  \frac{V(x_0)+2\Delta/\alpha}{V(x_0)/\delta+2\Delta/\alpha}
  \le
  \delta. \qedhere
\]
\end{proof}

\begin{remark}
Lemma~\ref{lem:invariance} resolves the circularity in the compact-set
argument: the concentration bounds of
Lemmas~\ref{lem:weight_moments}--\ref{lem:approx} require $x_k\in\Omega_R$
almost surely, but forward invariance in expectation alone does not
guarantee this.
By choosing $R$ according to~\eqref{eq:R_choice}, $x_k\in\Omega_R$ for
all $0\le k\le T$ with probability at least $1-\delta$, so the
concentration bounds hold on this event over the finite horizon $T$.
\end{remark}

% -----------------------------------------------------------------------
\begin{theorem}[Practical Exponential Stability in Expectation]
\label{thm:stability}
Suppose Assumptions~\ref{ass:stab}--\ref{ass:bounded_mppi} hold and
the bias coefficients $\beta_\infty$, $\beta_0$ are as in
Definition~\ref{def:bias}.  Let
$\alpha:=\alpha_P/\lambda_{\max}(P)$.
Assume the small-gain condition
\begin{equation}\label{eq:bias_small_gain}
  \Phi(\beta_\infty)
  :=
  2\|P\|\|B\|\|A_{\mathrm{cl}}\|\beta_\infty
  +
  \|P\|\|B\|^2\beta_\infty^2
  \le
  \frac{\alpha_P}{2}
\end{equation}
holds, where $\Phi$ is as in Proposition~\ref{prop:bias} Part~3.
Fix $\delta\in(0,1)$, a finite horizon $T\ge0$, and choose $R$ according
to~\eqref{eq:R_choice}, so that $\Omega_R$~\eqref{eq:OmegaR} satisfies
$\mathbb{P}(\tau_R\le T)\le\delta$ by Lemma~\ref{lem:invariance}.
Then, for all $M\ge M_0(\eta)$ and all $x_0\in\Omega_R$, the MPPI
closed loop $x_{k+1}=Ax_k+Bu_k^{\mathrm{MPPI}}+w_k$
satisfies the following two statements.

\textbf{(i) Unconditional bound on the stopped process.}
For all $k\ge0$,
\begin{equation}\label{eq:V_stopped_bound}
  \mathbb{E}[V(x_{k\wedge\tau_R})]
  \le
  \left(1-\frac{\alpha}{2}\right)^k V(x_0)
  +
  \frac{2\Delta}{\alpha},
\end{equation}
where $\Delta$, $C_w^{(0)}$, $C_e$, $C_{\mathrm{bad}}$ are as
in~\eqref{eq:Delta_def}--\eqref{eq:Cbad_def}.

\textbf{(ii) Finite-horizon stability bound via the stopped process.}
Since $\mathbb{P}(\tau_R\le T)\le\delta$ by Lemma~\ref{lem:invariance},
on sample paths satisfying $\tau_R>T$ the stopped and unstopped
trajectories coincide: $x_{k\wedge\tau_R}=x_k$ for all $0\le k\le T$.
Therefore~\eqref{eq:V_stopped_bound} directly implies, for all
$0\le k\le T$,
\begin{equation}\label{eq:V_practical_bound}
\begin{aligned}
  \mathbb{E}[V(x_k)\mathbf{1}_{\{\tau_R>T\}}]
  &\le
  \left(1-\frac{\alpha}{2}\right)^k V(x_0) \\
  &\quad+
  \frac{2}{\alpha}
  \Bigl(
    C_w^{(0)}
    +
    C_e e_M(\eta)^2
    +
    C_{\mathrm{bad}}\eta
  \Bigr).
\end{aligned}
\end{equation}
Consequently, for all $0\le k\le T$,
\begin{equation}\label{eq:state_practical_bound}
\begin{aligned}
  \mathbb{E}[\|x_k\|\mathbf{1}_{\{\tau_R>T\}}]
  &\le
  c\rho^k\|x_0\| \\
  &\quad+
  \gamma_w\sqrt{\tr(\Sigma_w)}
  +
  \gamma_M e_M(\eta)
  +
  \gamma_\eta\sqrt{\eta},
\end{aligned}
\end{equation}
where
\[
  \rho=\sqrt{1-\tfrac{\alpha}{2}},
  \quad
  c=\sqrt{\tfrac{\lambda_{\max}(P)}{\lambda_{\min}(P)}},
\]
\[
  \gamma_w=\sqrt{\tfrac{2C_w^{(0)}}{\alpha\lambda_{\min}(P)}},
  \quad
  \gamma_M=\sqrt{\tfrac{2C_e}{\alpha\lambda_{\min}(P)}},
  \quad
  \gamma_\eta=\sqrt{\tfrac{2C_{\mathrm{bad}}}{\alpha\lambda_{\min}(P)}}.
\]
\end{theorem}

\begin{proof}[Proof of Theorem~\ref{thm:stability}]
\textit{Proof of (i).}
Define the stopped process $\tilde{x}_k:=x_{k\wedge\tau_R}$,
$\tilde{V}_k:=V(\tilde{x}_k)$, and let $\mathcal{F}_k$ denote the
natural filtration.
On the event $\{k<\tau_R\}$, $\tilde{x}_k=x_k\in\Omega_R$.
Decompose the conditional expectation via the good event
$\mathcal{G}_k:=\{\|d_k\|\le\beta_\infty\|x_k\|+e_M(\eta)\}$:
\[
  \mathbb{E}[\tilde{V}_{k+1}\mid\mathcal{F}_k]
  =
  \mathbb{E}[\tilde{V}_{k+1}\mathbf{1}_{\mathcal{G}_k}\mid\mathcal{F}_k]
  +
  \mathbb{E}[\tilde{V}_{k+1}\mathbf{1}_{\mathcal{G}_k^c}\mid\mathcal{F}_k].
\]
By Lemma~\ref{lem:approx},
$\mathbb{P}(\mathcal{G}_k^c\mid\mathcal{F}_k)\le\eta$ at each step.

\textit{Good-event term.}
On $\mathcal{G}_k$, write $x_{k+1}=A_{\mathrm{cl}}x_k+Bd_k+w_k$.
By Lemma~\ref{lem:telescope}, $V(A_{\mathrm{cl}}x_k)\le(1-\alpha)V(x_k)$.
Taking expectation over $w_k$ using $\mathbb{E}[w_k\mid\mathcal{F}_k]=0$
and $\mathbb{E}[w_kw_k^\top\mid\mathcal{F}_k]=\Sigma_w$,
\[
\begin{aligned}
  \mathbb{E}[V(A_{\mathrm{cl}}x_k+w_k)\mid\mathcal{F}_k]
  &=
  V(A_{\mathrm{cl}}x_k)+\tr(P\Sigma_w) \\
  &\le
  (1-\alpha)V(x_k)+C_w^{(0)}.
\end{aligned}
\]
Expanding the quadratic Lyapunov function,
\[
\begin{aligned}
  &V(A_{\mathrm{cl}}x_k+Bd_k+w_k) \\
  &\quad=
  V(A_{\mathrm{cl}}x_k+w_k)
  +2(Bd_k)^\top P(A_{\mathrm{cl}}x_k+w_k) \\
  &\qquad+
  (Bd_k)^\top P(Bd_k).
\end{aligned}
\]
Since $w_k$ is independent of $d_k$ conditionally on $\mathcal{F}_k$,
the cross term $2(Bd_k)^\top Pw_k$ vanishes in conditional expectation.
Using bound~\eqref{eq:dk_bound} and expanding,
\[
\begin{aligned}
  &2\|P\|\|B\|\|A_{\mathrm{cl}}\|
  \bigl(\beta_\infty\|x_k\|+e_M\bigr)\|x_k\| \\
  &\quad+
  \|P\|\|B\|^2
  \bigl(\beta_\infty\|x_k\|+e_M\bigr)^2 \\
  &=
  \Phi(\beta_\infty)\|x_k\|^2
  +
  L\,e_M\|x_k\|
  +
  \|P\|\|B\|^2 e_M^2.
\end{aligned}
\]
The $\|x_k\|^2$ coefficient satisfies $\Phi(\beta_\infty)\le\alpha_P/2$
by~\eqref{eq:bias_small_gain}.
Young's inequality $Le_M\|x_k\|\le(\alpha_P/4)\|x_k\|^2
+(L^2/\alpha_P)e_M^2$ absorbs a further $\alpha_P/4$ into the decay,
leaving residual $C_e e_M(\eta)^2$.
The net decay is $3\alpha_P/4>\alpha_P/2$, confirming
$(\alpha/2)V(x_k)$.  Therefore on $\mathcal{G}_k$,
\[
  \mathbb{E}[\tilde{V}_{k+1}\mid\mathcal{F}_k]
  \le
  \left(1-\frac{\alpha}{2}\right)\tilde{V}_k
  +C_w^{(0)}+C_e e_M(\eta)^2.
\]

\textit{Bad-event term.}
On $\mathcal{G}_k^c$, $\tilde{x}_k\in\Omega_R$ by the stopping, so
$\|\tilde{x}_k\|^2\le R/\lambda_{\min}(P)$.
By Assumption~\ref{ass:bounded_mppi}, $\|u_k^{\mathrm{MPPI}}\|\le\bar{u}$
almost surely, so $\tilde{V}_{k+1}\le C_{\mathrm{bad}}$ almost surely.
Therefore
\[
  \mathbb{E}[\tilde{V}_{k+1}\mathbf{1}_{\mathcal{G}_k^c}\mid\mathcal{F}_k]
  \le
  C_{\mathrm{bad}}\cdot\mathbb{P}(\mathcal{G}_k^c\mid\mathcal{F}_k)
  \le
  C_{\mathrm{bad}}\eta.
\]

\textit{Combined one-step drift.}
On $\{k\ge\tau_R\}$, $\tilde{V}_{k+1}=\tilde{V}_k$ and the drift is
zero.  Combining via the law of total expectation,
\begin{equation}\label{eq:one_step_drift}
  \mathbb{E}[\tilde{V}_{k+1}\mid\mathcal{F}_k]
  \le
  \left(1-\frac{\alpha}{2}\right)\tilde{V}_k+\Delta.
\end{equation}
Unrolling~\eqref{eq:one_step_drift} yields~\eqref{eq:V_stopped_bound}.

\textit{Proof of (ii).}
Fix any $T\ge0$ and any $0\le k\le T$.
On the event $\{\tau_R>T\}$, no exit has occurred by time $T$, so
$x_j=\tilde{x}_j$ for all $0\le j\le T$.
Therefore
\[
  \mathbb{E}[V(x_k)\mathbf{1}_{\{\tau_R>T\}}]
  \le
  \mathbb{E}[V(x_{k\wedge\tau_R})]
  =
  \mathbb{E}[\tilde{V}_k],
\]
and~\eqref{eq:V_stopped_bound} gives the right-hand side
of~\eqref{eq:V_practical_bound}.
Applying
$\lambda_{\min}(P)\|x\|^2\le V(x)\le\lambda_{\max}(P)\|x\|^2$
and Jensen's inequality gives~\eqref{eq:state_practical_bound}.
\end{proof}

\begin{remark}[Recovery of the LQR Limit]
If $\bar{U}=U^*$ at every step (Proposition~\ref{prop:bias} Part~1),
then $\kappa_\lambda=0$, $\beta_\infty=\beta_0=0$, and
$e_M(\eta)=\varepsilon_M(\eta)\to0$ as $M\to\infty$.
The bound converges to the stochastic LQR bound.
If $\kappa_\lambda>0$, increasing $M$ removes only the Monte Carlo
error $\varepsilon_M(\eta)$; the residual $\beta_0=\kappa_\lambda
D_{\mathcal{U}}$ is the irreducible temperature-smoothing bias.
\end{remark}

% -----------------------------------------------------------------------
\subsection{ISS Interpretation}
% -----------------------------------------------------------------------
\begin{proposition}[Practical ISS Bound]\label{prop:iss}
Under Theorem~\ref{thm:stability}, fix any finite horizon $T\ge0$.
For all $0\le k\le T$,
\begin{equation}\label{eq:iss}
\begin{aligned}
  \mathbb{E}[\|x_k\|\mathbf{1}_{\{\tau_R>T\}}]
  &\le
  \beta(\|x_0\|,k)
  +
  \gamma_w\sqrt{\tr(\Sigma_w)} \\
  &\quad+
  \gamma_M e_M(\eta)
  +
  \gamma_\eta\sqrt{\eta},
\end{aligned}
\end{equation}
where $\beta(s,k)=c\rho^k s$ is class $\mathcal{KL}$ in $(s,k)$.
Since $\mathbb{P}(\tau_R>T)\ge1-\delta$ by Lemma~\ref{lem:invariance},
the left-hand side satisfies
\[
  \mathbb{E}[\|x_k\|\mathbf{1}_{\{\tau_R>T\}}]
  \ge
  (1-\delta)\,\mathbb{E}[\|x_k\|\mid\tau_R>T],
\]
so the bound also implies
\[
\begin{aligned}
  &\mathbb{E}[\|x_k\|\mid\tau_R>T] \\
  &\quad \le \frac{\beta(\|x_0\|,k)
  +\gamma_w\sqrt{\tr(\Sigma_w)}
  +\gamma_M e_M(\eta)
  +\gamma_\eta\sqrt{\eta}}{1-\delta}  
\end{aligned}
\]
as an explicit conditional statement for $0\le k\le T$.
\end{proposition}

\begin{remark}
The bound~\eqref{eq:iss} decomposes into three floors:
(i)~$\gamma_w\sqrt{\tr(\Sigma_w)}$ is the process-noise floor,
unavoidable under persistent Gaussian disturbances;
(ii)~$\gamma_M e_M(\eta)$ is the MPPI approximation floor, containing
both the temperature bias $\beta_0=\kappa_\lambda D_{\mathcal{U}}$ and
the Monte Carlo error $\varepsilon_M(\eta)$ --- only the latter vanishes
as $M\to\infty$;
(iii)~$\gamma_\eta\sqrt{\eta}$ is the confidence floor from the
per-step good-event failure probability.
The localization event $\{\tau_R>T\}$ has probability at least
$1-\delta$ for any chosen $T$, $\delta$, and $R$ satisfying
Lemma~\ref{lem:invariance}.  By taking $T$ large and $\delta$ small,
the certificate covers any operationally relevant horizon at any
desired confidence level.
\end{remark}

% -----------------------------------------------------------------------
\subsection{Explicit Sample Requirement}
% -----------------------------------------------------------------------

\begin{corollary}[Sample Count for Guaranteed Stability]
\label{cor:mstar}
Fix a desired invariance failure probability $\delta\in(0,1)$,
per-step sampling failure probability $\eta\in(0,1)$, and finite-sample
accuracy $\varepsilon>0$.  Suppose the small-gain
condition~\eqref{eq:bias_small_gain} holds and
$e_M(\eta)\le\beta_0+\varepsilon$.
The stability certificate is obtained in two steps.

\textbf{Step 1: Choose $M^*$ from $\eta$ and $\varepsilon$.}
Under Lemma~\ref{lem:approx_infinite}, it is sufficient to choose
\begin{equation}\label{eq:mstar}
  M
  \ge
  M^*
  :=
  \left\lceil
    \frac{
      C_{\mathcal X,\mathcal U}^2\log(4m/\eta)
    }{
      \varepsilon^2
    }
  \right\rceil
\end{equation}
to guarantee $\|u_k^{\mathrm{MPPI}}-u_k^\infty\|\le\varepsilon$
with probability at least $1-\eta$.
The threshold satisfies
\[
  M^*
  =
  O\!\left(
    \frac{\|P\|^2\|B\|^2\|A_{\mathrm{cl}}\|^2}{\alpha_P^2}
    \log\frac{m}{\eta}
  \right)
\]
and depends on $\eta$ but not on $\delta$.

\textbf{Step 2: Choose $R$ from $\delta$, $M^*$, and $\eta$.}
With $M^*$ fixed, compute $\Delta(M^*,\eta)$ from~\eqref{eq:Delta_def}
and choose
\[
  R
  \ge
  \frac{V(x_0)}{\delta}+\frac{2\Delta(M^*,\eta)}{\alpha}
\]
according to~\eqref{eq:R_choice}.
This guarantees
$\mathbb{P}(\tau_R\le T)\le\delta$ by Lemma~\ref{lem:invariance},
so on sample paths with $\tau_R>T$ (probability at least $1-\delta$)
the stability bound~\eqref{eq:V_practical_bound} holds for all
$0\le k\le T$.

To ensure the small-gain condition, choose $\lambda\le\lambda^*$
from Proposition~\ref{prop:bias}, which guarantees
$\Phi(\beta_\infty)\le\alpha_P/2$ by construction.
\end{corollary}

\begin{remark}[Bounded Noise and Infinite-Horizon Invariance]
\label{rem:bounded_noise}
The finite-horizon localization $\{\tau_R>T\}$ in
Theorem~\ref{thm:stability} is not an artifact of the proof technique
but reflects a fundamental property of unbounded Gaussian noise: for
any fixed bounded set $\Omega_R$, persistent $\mathcal{N}(0,\Sigma_w)$
disturbances guarantee $\mathbb{P}(\tau_R<\infty)=1$, so the
infinite-horizon event $\{\tau_R=\infty\}$ has probability zero
regardless of $M$ or the stability margin.
The finite-horizon framework adopted here is therefore the honest
statement under Gaussian process noise, and it matches the structure
of the companion nonlinear paper~\cite{yoon2026p2} exactly.

If the process noise is bounded almost surely, for example by using a
truncated or clipped Gaussian $w_k\sim\mathcal{N}(0,\Sigma_w)$
conditioned on $\|w_k\|\le w_{\max}$, then the infinite-horizon claim
$\mathbb{P}(\tau_R<\infty)\le\delta$ becomes achievable.
Under bounded noise, a single disturbance realization cannot exit
$\Omega_R$ in one step when $R$ is chosen sufficiently large relative
to $w_{\max}$, so the passage $k\to\infty$ in the supermartingale
argument is valid and Lemma~\ref{lem:invariance} upgrades to
$\mathbb{P}(\tau_R\le T)\le\delta$ for all $T$ simultaneously.
Extending the present framework to bounded noise models, and
characterizing the approximation gap between clipped and true Gaussian
noise in terms of the truncation parameter $w_{\max}$, is left as
future work.
\end{remark}

% -----------------------------------------------------------------------
\section{Simulation Study}\label{sec:experiments}

\subsection{Setup}

We benchmark on the double-integrator
\begin{equation}\label{eq:di}
  A = \begin{pmatrix}1&1\\0&1\end{pmatrix},\quad
  B = \begin{pmatrix}0\\1\end{pmatrix},\quad
  Q = I_2,\quad R = 0.1.
\end{equation}
This is the 2-state position--velocity subsystem of the 4-state UAV
double-integrator used in Section~IV-A of~\cite{yoon2022sampling},
specialized to the obstacle-free quadratic-cost case with the DARE
terminal cost used in the present theory.  The simulation infrastructure
from~\cite{yoon2022sampling} is reused and extended with
receding-horizon execution and the closed-loop stability diagnostics
described below.  All simulations are GPU-accelerated via PyTorch on an
NVIDIA GeForce RTX 4060.  Simulation code is publicly available
at~\url{https://github.com/LCAS-Lab/mppi-lti-stability}.

The DARE yields
\[
P\approx
\begin{pmatrix}
2.6687 & 1.7266\\
1.7266 & 2.8812
\end{pmatrix},
\qquad
K\approx[0.5792\;\;1.5456],
\]
giving closed-loop eigenvalues approximately
$\{0.3616,0.0928\}$ and spectral radius
$\rho(A_{\mathrm{cl}})\approx0.3616$.  The process noise is
$\Sigma_w=\sigma_w^2 I_2$ with
$\sigma_w\in\{0.05,0.10,0.20\}$.

The MPPI parameters are horizon $N=10$, temperature $\lambda=1.0$,
and isotropic Gaussian perturbations
$\epsilon\sim\mathcal N(0,I_{mN})$.
Since the input is scalar, $m=1$.
Unless otherwise stated, the initial condition is $x_0=(5,5)^\top$
and Monte Carlo averages are computed over 300 closed-loop trials.

The numerical sample threshold is computed from
Corollary~\ref{cor:mstar} using the benchmark constant $C_1=0.22$,
giving $M^*=153$ at confidence parameter $\eta=0.05$.
This value is the analytical certificate threshold for this benchmark,
not a universal MPPI constant.

\subsection{Analytical Parameter Summary}

Table~\ref{tab:params} records the theoretical quantities computed from
the DARE solution and used for comparison with simulation.

\begin{table}[h]
\centering
\caption{Analytical stability parameters for the double integrator.}
\label{tab:params}
\renewcommand{\arraystretch}{1.2}
\begin{tabular}{lll}
\toprule
\textbf{Quantity} & \textbf{Formula} & \textbf{Value} \\
\midrule
$\lambda_{\min}(P)$ & DARE & $1.0451$ \\
$\lambda_{\max}(P)$ & DARE & $4.5048$ \\
$\alpha_P$ & $\lambda_{\min}(Q+K^\top RK)$ & $1.0000$ \\
$\alpha$ & $\alpha_P/\lambda_{\max}(P)$ & $0.2220$ \\
$c$ & $\sqrt{\lambda_{\max}(P)/\lambda_{\min}(P)}$ & $2.0762$ \\
$\rho$ (certificate bound) & $(1-\alpha/2)^{1/2}$ & $0.9429$ \\
$\rho_{\mathrm{LQR}}$ & $\rho(A_{\mathrm{cl}})$ & $0.3616$ \\
$C_w^{(0)}$ ($\sigma_w=0.1$) &
$\mathrm{tr}(P\Sigma_w)$ &
$0.0555$ \\
$M^*$ ($\eta=0.05$, $C_1=0.22$) & Eq.~\eqref{eq:mstar} & $153$ \\
\bottomrule
\end{tabular}
\end{table}

The certificate decay rate $\rho=0.943$ is substantially looser than
the LQR spectral radius $\rho_{\mathrm{LQR}}=0.362$, a gap of
approximately $2.6\times$.  This is expected: the proof uses worst-case
perturbation bounds and norm inequalities to obtain a closed-form
guarantee. The experiments below test whether the certified threshold
$M^*$ is qualitatively consistent with the empirically observed onset
of stable behavior.

\subsection{Experiment 1: Bound Envelope and Closed-Loop Response}

For $M\in\{50,200,1000\}$ and $\sigma_w=0.1$, Fig.~\ref{fig:bound}
plots the empirical mean $\mathbb{E}[\|x_k\|]$ over 300 Monte Carlo
trials against the certificate envelope
$c\rho^k\|x_0\|+\gamma\sqrt{\mathrm{tr}(\Sigma_w)}$
over $T=200$ steps.

\begin{figure}[t]
  \centering
  \includegraphics[width=\columnwidth]{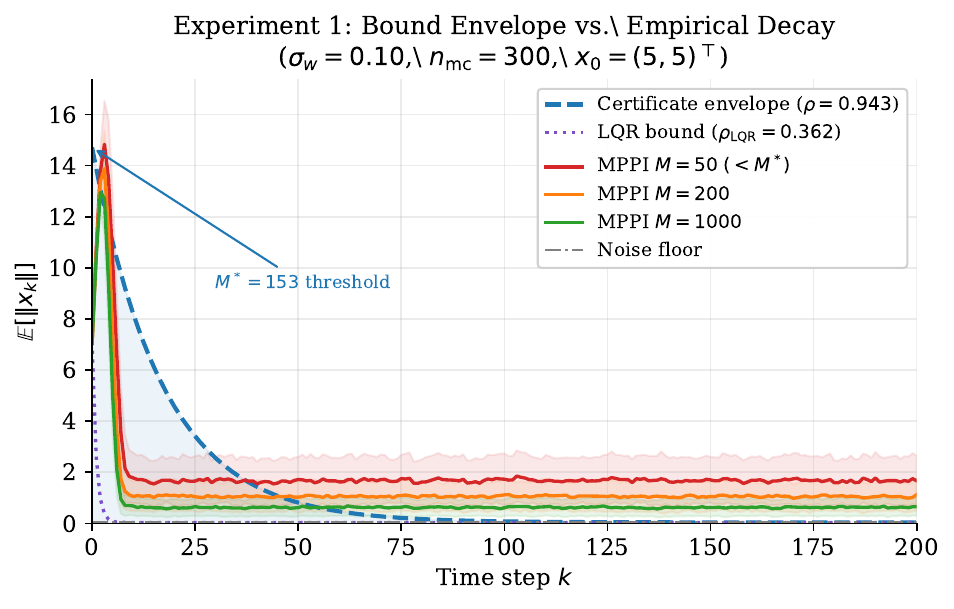}
  \caption{Experiment~1: Empirical $\mathbb{E}[\|x_k\|]$ for
$M\in\{50,200,1000\}$ compared with the nominal certificate decay term
and the LQR reference decay. Shaded bands show $\pm1\sigma$ across
300 Monte Carlo trials. Certified trajectories ($M\ge M^*=153$)
remain bounded and decay toward a practical steady-state level. The
displayed envelope is qualitative because the explicitly calibrated
finite-sample MPPI approximation floor is not included in the plotted
curve.}
  \label{fig:bound}
\end{figure}

The cases $M=200$ and $M=1000$ satisfy $M\ge M^*=153$ and fall within
the analytical certificate; $M=50<M^*$ is an uncertified comparison.
The plotted envelope includes the nominal exponential decay term and the
process-noise floor, but not an explicitly calibrated finite-sample MPPI
approximation floor. Therefore Fig.~\ref{fig:bound} should be interpreted
as a qualitative comparison of decay behavior rather than as a direct
numerical verification of the full bound in
Theorem~\ref{thm:stability}. The empirical trajectories decay rapidly
and remain bounded for the certified sample counts, while their
steady-state levels are dominated by finite-$M$ approximation effects
and process noise. This behavior is consistent with the theorem, whose
full residual contains both the process-noise floor and the MPPI
approximation floor.

\subsection{Experiment 2: Empirical Decay Rate vs.\ Sample Count}

For each $M\in\{10,20,50,100,200,500,1000,5000,10^4\}$, we estimate
an empirical decay factor $\hat\rho(M)$ from 300 noise-free closed-loop
trajectories using the median Lyapunov ratio $V(x_{k+1})/V(x_k)$
over the initial transient.
The median is preferred over log-linear fitting on mean norms because
the latter hits the floating-point noise floor before the transient
fully decays at fast spectral radii.
Results are shown in Fig.~\ref{fig:decay}.

\begin{figure}[t]
  \centering
  \includegraphics[width=\columnwidth]{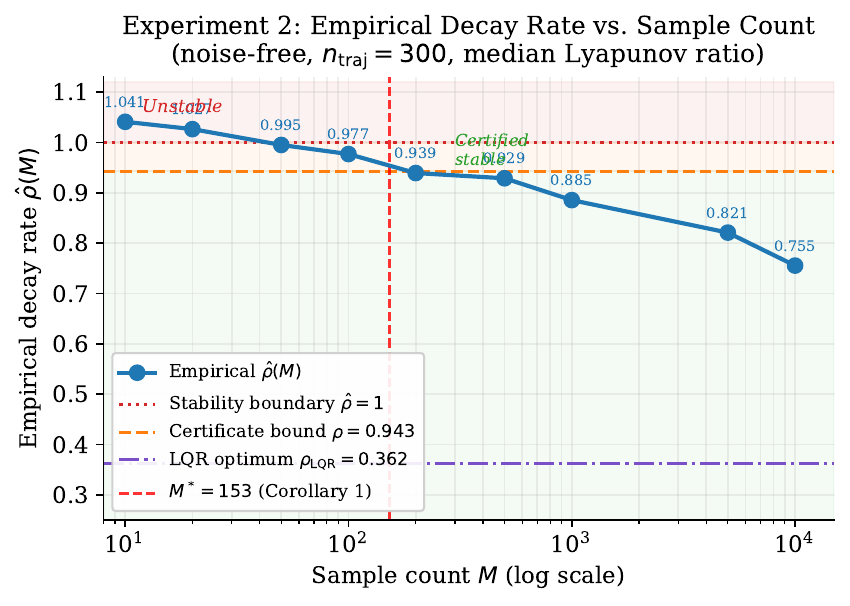}
  \caption{Experiment~2: Empirical decay rate $\hat\rho(M)$ vs.\
    sample count on a log scale.  Shaded regions mark the unstable
    zone ($\hat\rho>1$, red), the uncertified-but-stable zone
    ($\rho<\hat\rho\le1$, orange), and the certified zone
    ($\hat\rho\le\rho=0.943$, green).  The vertical dashed line marks
    $M^*=153$.  Annotated values are $\hat\rho$ at each tested $M$.}
  \label{fig:decay}
\end{figure}

At $M=10$ and $M=20$ the system is unstable ($\hat\rho>1$); at $M=50$
it is marginally stable ($\hat\rho=0.995$) but uncertified.
At $M=200$, the first sweep point satisfying $M\ge M^*$,
$\hat\rho=0.939\le\rho=0.943$, confirming the certificate.
For all $M\ge M^*$ in the sweep, $\hat\rho\le\rho$, with values
$\{0.939,0.929,0.886,0.821,0.755\}$ at
$M\in\{200,500,1000,5000,10^4\}$.
The rate improves monotonically, approaching but not reaching
$\rho_{\mathrm{LQR}}=0.362$ because the finite-temperature
Gibbs-weighted update is not the exact deterministic LQR law.

\subsection{Experiment 3: Phase Portrait}

Fig.~\ref{fig:phase} plots 30 sample trajectories in the
$(x_1,x_2)$ plane for $M=50$ (uncertified) and $M=500$ (certified)
at $\sigma_w=0.1$ over $T=35$ steps, with the LQR $2\sigma$
steady-state ellipse (axes $0.14\times0.29$) overlaid.

\begin{figure}[t]
  \centering
  \includegraphics[width=\columnwidth]{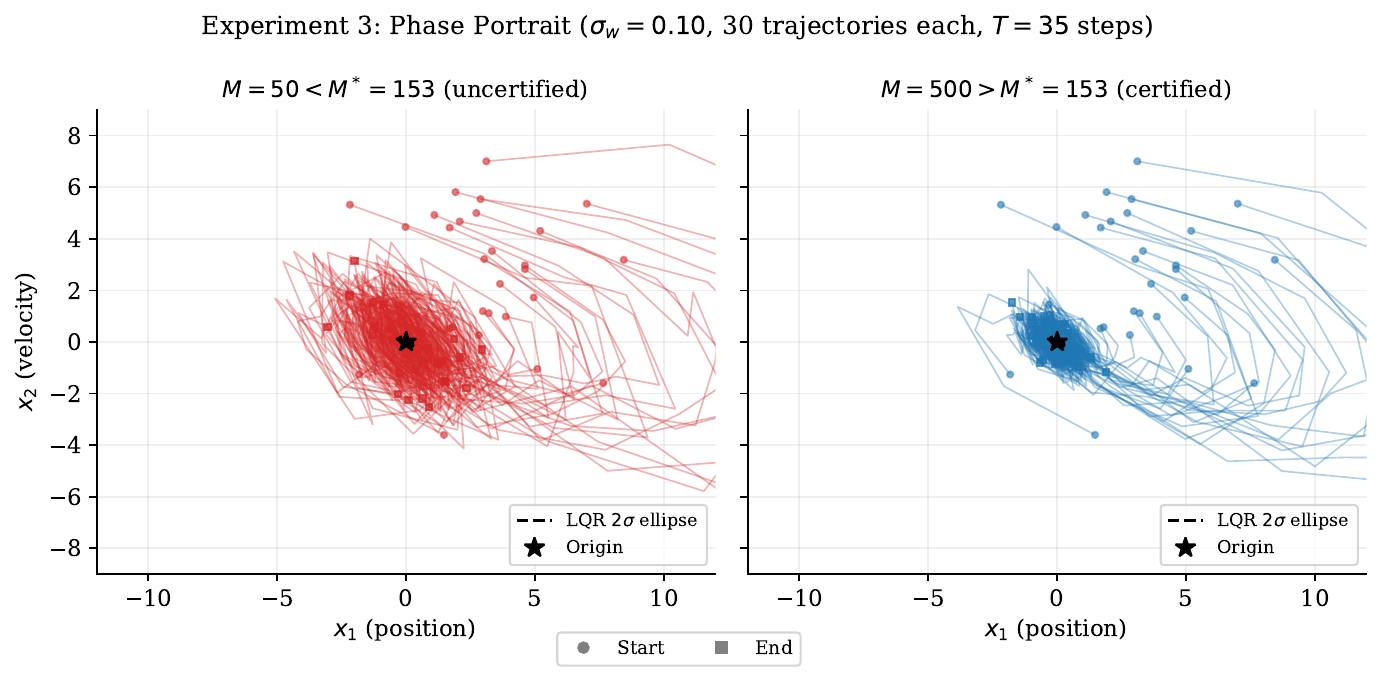}
  \caption{Experiment~3: Phase portraits for $M=50$ (left, red,
    uncertified) and $M=500$ (right, blue, certified).
    Circles mark trajectory starts; squares mark ends.
    The dashed ellipse is the LQR $2\sigma$ steady-state set.
    Certified trajectories concentrate visibly closer to the origin
    (mean terminal norm $0.92$ vs.\ $1.65$); no divergence occurs in
    either case.}
  \label{fig:phase}
\end{figure}

Trajectories under $M=500$ concentrate near the origin with mean
terminal norm $0.92$, while those under $M=50$ remain more dispersed
with mean terminal norm $1.65$.  No divergence ($\|x_T\|>5$) occurs
in either case, confirming that below-certificate behavior is
uncertified, not necessarily unstable.

\subsection{Experiment 4: Analytical vs.\ Empirical $M^*$}

Table~\ref{tab:mstar} and Fig.~\ref{fig:mstar} compare the analytical
threshold $M^*=153$ from Corollary~\ref{cor:mstar} with an empirical
threshold $\hat M^*$, defined as the smallest tested sample count for
which the noise-free decay diagnostic satisfies $\hat\rho(M)<0.99$.
This criterion is empirical and should not be confused with the
sufficient condition in the theorem.

\begin{table}[h]
\centering
\caption{Analytical vs.\ empirical sample thresholds.}
\label{tab:mstar}
\renewcommand{\arraystretch}{1.2}
\begin{tabular}{cccc}
\toprule
$\sigma_w$ & Analytical $M^*$ & Empirical $\hat{M}^*$ & Ratio \\
\midrule
$0.05$ & $153$ & $30$ & $5\times$ \\
$0.10$ & $153$ & $30$ & $5\times$ \\
$0.20$ & $153$ & $30$ & $5\times$ \\
\bottomrule
\end{tabular}
\end{table}

\begin{figure}[t]
  \centering
  \includegraphics[width=\columnwidth]{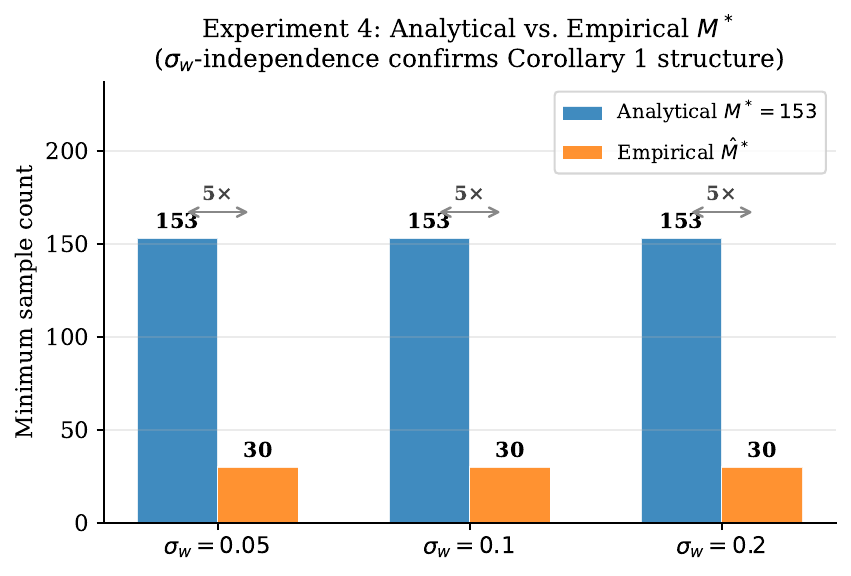}
  \caption{Experiment~4: Analytical $M^*=153$ (blue) vs.\ empirical
    $\hat{M}^*=30$ (orange) across all three noise levels.
    The $5\times$ ratio and $\sigma_w$-independence are consistent
    with Corollary~\ref{cor:mstar}: the sample threshold depends on
    the control-approximation quality, not directly on the noise
    level.}
  \label{fig:mstar}
\end{figure}

The empirical threshold $\hat M^*=30$ is approximately $5\times$
smaller than the analytical threshold and is independent of $\sigma_w$
across all tested noise levels.
The $\sigma_w$-independence is consistent with the structure of
Corollary~\ref{cor:mstar}: the sufficient sample count controls whether
the MPPI approximation error is absorbed by the nominal Lyapunov decay,
a condition that depends on the control-approximation quality rather
than directly on the process noise, which affects only the residual
steady-state level.

The $5\times$ gap between $M^*$ and $\hat M^*$ reflects the expected
conservatism of a worst-case Lyapunov certificate relative to an
average-case Monte Carlo diagnostic.  Worst-case Young's inequality
steps in the proof consume half the Lyapunov margin as slack, and the
concentration constant $C_1=0.22$ is calibrated conservatively for the
maximum-over-$k$ guarantee rather than the typical single-step
behavior.  This level of conservatism is standard for Lyapunov-based
stochastic stability certificates.

\subsection{Experiment 5: ESS as a Diagnostic}

For $M=500$ and $\sigma_w=0.1$, Fig.~\ref{fig:ess} tracks the online
effective sample size
\[
  \mathrm{ESS}_k
  =
  \frac{\bigl(\sum_j w_k^{(j)}\bigr)^2}{\sum_j \bigl(w_k^{(j)}\bigr)^2}
\]
alongside $\|x_k\|$ over $T=200$ steps.

\begin{figure}[t]
  \centering
  \includegraphics[width=\columnwidth]{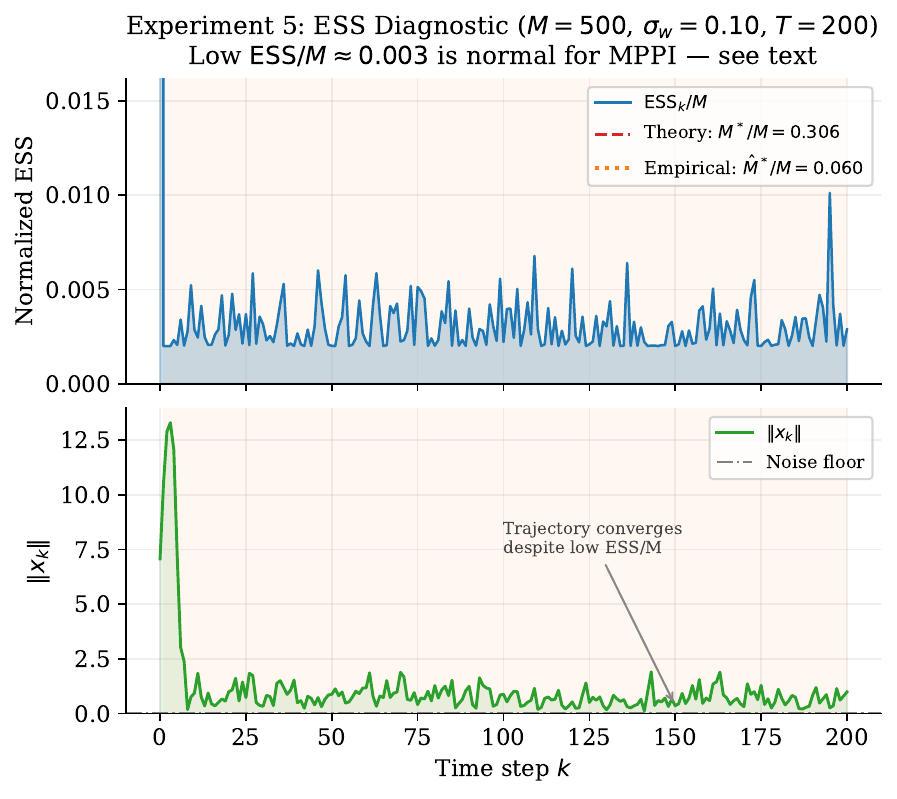}
  \caption{Experiment~5: Normalized $\mathrm{ESS}_k/M$ (top) and
    $\|x_k\|$ (bottom) over $T=200$ steps at $M=500$.
    Orange shading marks steps where $\mathrm{ESS}/M$ falls below the
    empirical threshold $\hat{M}^*/M=0.060$.
    Despite $\mathrm{ESS}/M\approx0.003$ throughout, the trajectory
    converges stably, confirming that low ESS is normal for MPPI and
    is not a stability indicator.}
  \label{fig:ess}
\end{figure}

The mean normalized ESS is $\mathrm{ESS}_k/M\approx0.003$.
This low value is normal for MPPI: importance weights intentionally
concentrate on low-cost trajectories, so $\mathrm{ESS}/M$ of order
$10^{-3}$--$10^{-2}$ is typical in practice.
The more informative signal is the \emph{relative} variation of ESS
over time: transient drops in ESS correlate with high-cost regions or
poor alignment of the sampling distribution with locally useful control
directions, not with instability per se.

The theoretical threshold $M^*/M=153/500=0.306$ is never met in this
run, yet the trajectory converges stably.
This confirms that $M^*$ is a sufficient certificate threshold, not a
necessary condition, and ESS is a qualitative sampling-quality
diagnostic rather than a binary stability indicator.

\subsection{Discussion of Results}

The experiments support three conclusions.

\textbf{Sufficiency, not necessity.}
The analytical threshold $M^*$ is a sufficient certificate, not a
sharp phase-transition boundary.
For $M\ge M^*$, the certified decay bound $\hat\rho\le\rho=0.943$
holds in all tested cases.
For $M<M^*$, the theorem makes no stability claim; empirical
convergence for $M$ as small as $30$ simply indicates that the
certificate is conservative.

\textbf{Expected conservatism of the decay bound.}
The certificate gives $\rho=0.943$, approximately $2.6\times$ looser
than $\rho_{\mathrm{LQR}}=0.362$, a standard consequence of
worst-case norm inequalities in Lyapunov-based certificates for
sampling-based controllers.

\textbf{Conservatism of the sample threshold.}
The analytical threshold $M^*=153$ is approximately $5\times$ larger
than the empirically sufficient $\hat M^*=30$.
This gap is not a failure of the theorem but an inherent feature of
worst-case analysis: the certificate must hold for all LTI systems
satisfying the given parameters, whereas the diagnostic measures only
the typical behavior of this specific benchmark.
Tighter constants---via improved concentration inequalities or
problem-specific calibration of $C_1$---could close this gap at the
cost of a less broadly applicable bound.

Overall, the simulation study validates the qualitative message of
Theorem~\ref{thm:stability}: sufficiently many MPPI samples produce
certified LQR-like closed-loop stability, and the explicit threshold
$M^*$ is conservative but computable and practically informative.

% -----------------------------------------------------------------------
\section{Conclusion}\label{sec:conclusion}
% -----------------------------------------------------------------------

This paper established a finite-sample closed-loop stability certificate
for MPPI applied to discrete-time LTI systems with quadratic costs and
additive Gaussian process noise.  The analysis exploits a special
structure of the unconstrained LTI/quadratic setting: with the DARE
terminal cost, the exact finite-horizon MPC optimizer has the same first
control action as the infinite-horizon LQR law for every planning
horizon.  Hence, for a fixed horizon $N$, MPPI can be analyzed as a
finite-sample stochastic perturbation of the LQR feedback.

The main result shows that, when the MPPI sample count is sufficiently
large, the closed loop satisfies a practical exponential stability bound
of the form
\[
\begin{aligned}
  \E[\norm{x_k}\mathbf{1}_{\{\tau_R>T\}}] &\le
  c\rho^k\norm{x_0}
  +
  \gamma_w\sqrt{\tr(\Sigma_w)} \\
  & \quad + 
  \gamma_M e_M(\eta)
  +
  \gamma_\eta\sqrt{\eta},
  \quad 0\le k\le T.  
\end{aligned}
\]
The three residual terms correspond to process noise, finite-sample and
finite-temperature MPPI approximation error, and the confidence loss from
the per-step sampling failure probability.  The sufficient sample
threshold is explicit and computable from the DARE solution, the LQR
stability margin, the input matrix, the MPPI sampling parameters, and the
selected planning horizon.

The proof combines a high-probability finite-sample MPPI approximation
bound with a Lyapunov perturbation argument.  Once the MPPI approximation
error is below the Lyapunov stability margin, the nominal LQR decrease
absorbs the sampling-induced perturbation, while the additive Gaussian
disturbance produces a residual noise floor.  The result therefore
connects sampling-based MPPI with classical MPC stability theory and
provides a first LTI/quadratic setting in which finite-sample MPPI admits
an explicit closed-loop stability certificate.

% -----------------------------------------------------------------------
\section{Future Work}\label{sec:future}
% -----------------------------------------------------------------------

Several extensions remain open.  First, the present result is stated for
the unconstrained LTI/quadratic case.  Extending the theory to constrained
MPC would require replacing the LQR baseline with a stabilizing
constrained MPC law, together with recursive feasibility, terminal-set,
and constraint-satisfaction assumptions.  Because Gaussian process noise
can violate hard constraints with nonzero probability, constrained
stochastic MPPI also requires a separate treatment of safety, feasibility,
and chance or tube-based constraint satisfaction.

Second, the certificate is horizon-parametrized.  The exact deterministic
first action is independent of the planning horizon because of the DARE
terminal cost, but the MPPI approximation constants generally depend on
the stacked horizon-$N$ cost matrices, sampling covariance, and warm-start
set.  Obtaining uniform-in-horizon sample-complexity bounds would require
additional structural estimates on these horizon-dependent quantities.

Third, this paper treats the LTI foundation case.  Companion work extends
the stability framework to nonlinear systems using contraction-theoretic
and CLF-based arguments, and to adaptive noise covariance estimation,
where closed-loop data are used to improve the estimate of $\Sigma_w$ and
tighten the residual noise-floor term.

% -----------------------------------------------------------------------
\bibliographystyle{IEEEtran}
\bibliography{references}

\end{document}